\documentclass[12pt]{amsart}
\usepackage{amsmath, amsfonts, amssymb, latexsym, amsthm, amscd, mathrsfs, stmaryrd} 
\usepackage[all]{xy}
\usepackage[makeroom]{cancel}
\usepackage{yhmath}
\usepackage{hyperref}

\theoremstyle{definition}
\newtheorem{Question}[subsubsection]{Question}
\newtheorem{rem}[subsubsection]{Remark}
\theoremstyle{plain}
\newtheorem{prop}[subsubsection]{Proposition}
\newtheorem{thm}[subsubsection]{Theorem}
\newtheorem{thrm}{Theorem}

\newtheorem{cor}[subsubsection]{Corollary}

\newcommand{\mbf}{\mathbf}
\newcommand{\mbb}{\mathbb}
\newcommand{\mrm}{\mathrm}

\usepackage{color}
\newcommand{\nc}{\newcommand}
\nc{\redtext}[1]{\textcolor{red}{#1}}
\nc{\bluetext}[1]{\textcolor{blue}{#1}}
\nc{\greentext}[1]{\textcolor{green}{#1}}
\nc{\yl}[1]{\redtext{ #1}}
\nc{\zb}[1]{\redtext{From zb: #1}}

\newcommand{\A}{\mathcal A}

\newcommand{\GL}{\mrm{GL}}

\renewcommand{\H}{\mbf H}
\newcommand{\U}{\mbf U}
\newcommand{\V}{\mbf V}

\newcommand{\mult}{\mrm{mult}}
\newcommand{\can}{\mrm{can}}
\newcommand{\aff}{\mrm{aff}}
\newcommand{\ve}{\varepsilon}

\newcommand{\T}{\mbf T}

\title[Hecke algebras and edge contractions]{Hecke algebras and edge contractions}

\author{Yiqiang Li}
\address{Department of Mathematics\\ University at Buffalo\\  State University of New York  \\Buffalo, NY 14260}
\email{yiqiang@buffalo.edu}

\date{\today}

\begin{document}

\begin{abstract}
We establish an embedding from  the Hecke algebra associated with the edge contraction of a Coxeter system along an edge 
to  the Hecke algebra associated with the original Coxeter system. 
\end{abstract}

\maketitle

\section*{Introduction}
 
 \subsection{Edge contraction and representation theory}
Edge contraction is a fundamental operation on graphs, where two vertices are merged along an edge, producing a new graph. Investigating the relationship between representation-theoretic objects attached to a graph and those to the contracted graph is of considerable interest. Such investigations, besides their natural appeal, often have crucial applications. The behavior of various algebraic structures under edge contractions has been studied extensively. In \cite{Li23}, the behaviors of Weyl groups, Hall algebras, quantum groups, and Chevalley groups under edge contractions were examined, revealing a sub-quotient phenomenon, where the object attached to the contracted graph is a sub-quotient of the object attached to the original graph. Similar observations were made regarding the critical cohomological Hall algebras in \cite{LR24}. For Yangians of type A, Ueda provided further insights in a recent work \cite{U23}. It remains a natural question whether such a sub-quotient phenomenon holds for other representation-theoretic objects.

In this paper, we focus on the behavior of Iwahori-Hecke algebras under edge contractions. This study is a natural extension of the investigation in \cite{Li23}, as these algebras are quantum deformations of Weyl groups. Moreover, Hecke algebras can be associated with Coxeter systems, a notion more general than a graph (and a root datum). While the notion of edge contraction can be generalized to Coxeter matrices/systems, we demonstrate that the Coxeter group of a contracted Coxeter matrix is a subgroup of the Coxeter group of the original Coxeter matrix, establishing an explicit embedding. Furthermore, we show that this embedding extends naturally to Hecke algebras, showcasing the sub-quotient phenomenon.

The embedding of Hecke algebras can be described as follows. 
Let $(W_{S/e}, S/e)$ be the edge contraction of a Coxeter system $(W, S)$ along the edge $e=\{s_+, s_-\}$. 
Thus $S/e=S- \{ s_{+}, s_-\} + \{s_0\}$. 
Let $H_{S/e}$ and
$H_S$ be the respective Iwahori-Hecke algebras, with standard generators $T'_s$ and $T_s$, respectively.
 
\begin{thrm}[Theorem~\ref{H-inj}, Proposition~\ref{phi-H-ext}]
There is an embedding of algebras 
\[
H_{S/e} \hookrightarrow H_{S}, T'_s\mapsto 
\begin{cases}
T_s & \mbox{if}\  s \neq s_0,\\
T_{s_+} T_{s_-} T^{-1}_{s_+}  &\mbox{if}\ s= s_0.
\end{cases}
\]
\end{thrm}

\subsection{Proof techniques and linear trees}
The injective statement for Coxeter groups is proven via a calculus of the associated reflection representations, while on the Hecke algebra level, it is established through a specialization argument built upon the embeddings on Coxeter groups.

When edge contraction is performed along a linear tree, we demonstrate that these embeddings are conjugations of a naive embedding, akin to the embeddings of quantum groups discussed in \cite{Li23}.

\subsection{Extensions to affine structures}
The embeddings on Coxeter groups extend naturally to embeddings on affine Weyl groups. On the affine Hecke algebra level, we provide explicit embeddings under both the Iwahori-Matsumoto presentation and the Bernstein-Lusztig presentation.
While the embeddings of extended affine Hecke algebras are established in type A, for other types, further investigation is required.

\subsection{Compatibility and comparisons}
We establish that the embeddings of Hecke algebras are compatible with the embeddings of quantum $\mathfrak{sl}_n$ under Schur-Weyl duality, as explored in \cite{Li22, Li23}. On the Yangian level, comparisons are made in \cite{U24}.

\subsection{Geometric insights and future directions}
Given the geometric realization of affine Hecke algebras, we anticipate that the embeddings established in this paper will aid in analyzing edge contractions in related geometric settings. Exploring connections between these embeddings and results in \cite{M15, M18} on Khovanov-Lauda-Rouquier algebras could provide further insights.

\subsection{Relation to existing work}

M\"{u}hlherr studied in~\cite{Mu93} what type of Coxeter groups can be embedded in a fixed Coxeter group. 
Some results in {\it loc. cit.} were rediscovered  recently by Elias and Heng  as a Coxeter embedding for a Coxeter partition in \cite{EH24}. While the embeddings presented in this paper share similarities with those, they are not a special case thereof. Investigating a unified theory encompassing both types of embeddings promises to be an intriguing direction for future research.

 \subsection{}
 The layout of the paper is as follows. 
 In Section~\ref{Coxeter}, we recall the Coxeter system. We define the edge contraction operation on a Coxeter system. 
 We establish the embedding of the Coxeter groups and its affine variant.
 In Section~\ref{Hecke}, we establish the embeddings of Hecke algebras under an edge contraction.
 We relate the embedding with a naive embedding when the edge contraction is operated along a linear tree. 
 Section~\ref{HA} is devoted to an embedding of extended affine Hecke algebras of type A. 
 Section~\ref{AffH} is devoted to the embeddings of affine Hecke algebra under the Bernstein-Lusztig presentation.
 Section~\ref{Schur} is a study of the compatibility of the embeddings of Hecke algebras with the embeddings of quantum  $\mathfrak{sl}_n$. 
 The effect of edge contraction on a labelled edge is studied in the Appendix. 
  
\subsection{Acknowledgements}
It is a pleasure to thank Xuhua He for an interesting conversation. 
He asked   whether there is an embedding of Hecke algebras under edge contraction along a labelled edge, which indeed has an affirmative answer given
in the Appendix.  

I am  grateful to the anonymous referee for his/her suggestions that strengthen results in the paper. 
The referee made an observation that the image of the embedding on the Coxeter group level is  a reflection subgroup and hence Proposition~\ref{W-edge} can be deduced from Dyer's work 
~\cite{Dy87} and Deodhar's work~\cite{De89}. He/She also    pointed out the  relationship between the works~\cite{Mu93} and~\cite{EH24}.  
 The results in the Appendix were obtained under his/her suggestion and a question by Xuhua He. 
\tableofcontents

\section{Coxeter systems and edge contractions}
\label{Coxeter}

Let $\mbb Z$ be the set of integers and $\mbb Z_{\geq 1}$ be the set of natural numbers. 

\subsection{Coxeter system}
\label{Coxeter-1}
Let $S$ be a nonempty set. A Coxeter matrix $M=(m_{s, s'})_{s, s'\in S}$ is a symmetric matrix with entries in $\mbb Z_{\geq 1} \sqcup \{\infty\}$
such that
\[
m_{s, s'} =1 \ \mbox{iff} \ s=s'.
\]
Let $W_S\equiv W_{S, M}$ be the Coxeter group generated by elements in $S$ subject to the defining relation
$(ss')^{m_{s,s'}} = 1$, for all  $m_{s, s'} \neq \infty$.
The pair $(W_S, S)$ is a Coxeter system. 

\subsection{A linear representation}
\label{linear}
Given the pair $(S, M)$, we 
fix a matrix $K=(k_{s, s'})_{s, s' \in S}$ satisfying the following conditions.
\begin{align}
\label{K}
\begin{cases}
k_{s,s}=-2, &\forall s\in S;\\
k_{s, s'}=0, & \forall m_{s,s'}=2;\\
k_{s, s'}> 0, & \forall m_{s, s'} \geq 3; \\
k_{s, s'} k_{s', s} = 4 \cos^2\frac{\pi}{m_{s, s'}}, & \forall m_{s, s'} \neq \infty;\\
k_{s, s'} k_{s', s} \geq 4, &  \forall m_{s, s'}=\infty.
\end{cases}
\end{align}
One of the examples of $K$ is to take $k_{s, s'} = 2 \cos \frac{\pi}{m_{s, s'}}$ with the convention $\frac{\pi}{m_{s,s'}}=0$ if $m_{s, s'}=\infty$. 
Consider the real vector space $V_S= \oplus_{s\in S} \mbb R \alpha_s$ with the basis $\{ \alpha_s|s\in S\}$.
For each $s\in S$, define an endomorphism $\sigma_s: V_S\to V_S$ by
\[
\sigma_s (\alpha_{s'}) = \alpha_{s'}+ k_{s, s'} \alpha_s, \forall s'\in S. 
\]
It is well-known that the assignment $s\mapsto \sigma_s$ defines an embeddeing
$\sigma: W_S\to \GL(V_S)$. 

\subsection{Edge contraction}
\label{Edge}

Fix  a pair $e=\{s_+, s_-\} $ in $S$ such that $m_{s_+, s_-}=3$. 
Fix a symbol  $s_0$, which is not in $S$. 
We define $S/e= S-\{ s_+, s_-\} +\{ s_0\}$. 
We define a symmetric matrix  $N=(n_{s,s'})_{s, s'\in S/e}$ by setting $n_{s, s}=1$ for all $s$ and if $s\neq s'$, we set
\[
n_{s, s'} =
\begin{cases}
m_{s, s'} & \mbox{if} \ s, s'  \neq s_0,\\
m_{s, s_+} + m_{s, s_-} -2  & \mbox{if} \ s'=s_0 \ \mbox{and}, m_{s, s_+}= 2\ \mbox{or}\ m_{s, s_-}=2, \\
\infty & \mbox{if} \ s'=s_0, m_{s, s_+}, m_{s, s_-} > 2.
\end{cases}
\]
Clearly $N$ is a Coxeter matrix. 
We call the pair $(S/e, N)$ the edge contraction of $(S, M)$ along $e$.

\subsection{Embedding under edge contraction}

Let $W_{S/e, N}$ be the Coxeter group associated to the pair $(S/e, N)$. 
We have 

\begin{prop}
\label{W-edge}
There is an embedding $\phi: W_{S/e, N} \to W_{S, M}$ defined by 
\[
s\mapsto s, \mbox{if} \ s\in S-\{s_{\pm}\}, \mbox{and} \ s_0 \mapsto s_+s_-s_+.
\]
\end{prop}

\begin{proof}
By definition, we see that $\phi$ is a group homomorphism. 

We need to show that $\phi$ is injective. 
By assumption, we  consider the matrix $K$ relative to $(S, M)$ defined by  
\begin{align}
\label{K-edge}
k_{s, s'}= 2 \cos \frac{\pi}{m_{s,s'}}, \ \forall s, s'\in S.
\end{align}
We define  a matrix $\tilde K $ relative to $(S/e, N)$ by $\tilde k_{s, s'} = 2 \cos \frac{\pi}{n_{s, s'}}$ for all $s, s'\in S/e$. In particular, we have 
\begin{align}
\begin{cases}
\tilde k_{s, s'} = k_{s, s'} , & \mbox{if}\ s, s' \neq s_0, \\
\tilde k_{s, s_0}=k_{s, s_+} +k_{s, s_-} , & \mbox{if} \ s\neq s_0,\\
\tilde k_{s_0, s} = k_{s_+, s} +k_{s_-, s},&  \mbox{if} \ s\neq s_0.
\end{cases}
\end{align}
Recall the linear representation $V_S$ of $W_{S, M}$ defined via the matrix $K$ with the additional assumption (\ref{K-edge}) and   
the linear representation $V_{S/e}$ of $W_{S/e, N}$ via the above matrix $\tilde K$.
To avoid ambiguity, we write the associated embedding $\tilde \sigma: W_{S/e, N}\to \GL(V_{S/e})$, $s\mapsto \tilde \sigma_s$. 
By definition, we see that $V_{S/e}$ is a subspace of $V_S$ via the assignment
$\alpha_s\mapsto \alpha_s$, if $s\in S/e-\{s_0\}$, and $\alpha_{s_0}\mapsto \alpha_{s_+} +\alpha_{s_-}$. 
Let $W^{S/e}_{S, M}$ be the subgroup of $W_{S, M}$ generated by $s\in S-\{ s_{\pm}\}$ and $s_+ s_- s_+$. 
For any $s\in S$, we have 
\begin{align}
\begin{split}
\sigma_{s_+} \sigma_{s_-} \sigma_{s_+} (\alpha_{s}) 
& = \sigma_{s_+} \sigma_{s_-} ( \alpha_{s} + k_{s_+, s} \alpha_{s_+}) \\
& = \sigma_{s_+} (\alpha_s + k_{s_-, s} \alpha_{s_-} + k_{s_+, s} ( \alpha_{s_+}+\alpha_{s_-})) \\
& = \alpha_s + k_{s_+, s} \alpha_{s_+} + k_{s_-, s} (\alpha_{s_-}+\alpha_{s_+}) + k_{s_+, s} ( - \alpha_{s_+} + \alpha_{s_-}+\alpha_{s_+})\\
&= \alpha_s + \tilde k_{s_0, s} \alpha_{s_0}.
\end{split}
\end{align}
This implies that $\sigma_{s_+} \sigma_{s_-}\sigma_{s_+}$ leaves the subspace $V_{S/e}$ stable and when restricts to $V_{S/e}$, it coincides with the reflection $\tilde \sigma_{s_0}$ in $\GL(V_{S/e})$. 
It is clearly by definition that $\sigma_s$ for all $s\in S-\{s_{\pm}\}$ leaves $V_{S/e}$ stable and when restricts to $V_{S/e}$, it coincides with
the reflection $\tilde \sigma_{s}$ in $\GL(V_{S/e})$. 
Let $P$ be the parabolic subgroup of $\GL(V_S)$ consisting of all elements that leave $V_{S/e}$ stable. 
Then there exists a group homomorphism $\pi: P\to \GL(V_{S/e})$ via restriction $g\mapsto g|_{V_{S/e}}$. 
By composing with this group homomorphism, we have a group homomorphism $W^{S/e}_{S, M} \to \GL(V_{S/e})$ defined by $s\mapsto \sigma_s|_{V_{S/e}}$. 
Moreover, its image is exactly the image of $W_{S/e, N}$ under $\tilde \sigma$. 
Hence we have a surjective map $\psi: W^{S/e}_{S, M} \to W_{S/e, N}$ defined by $s\mapsto s$ and $s_+s_-s_+\mapsto s_0$.
Therefore, we have the following commutative diagram.
\[
\begin{CD}
W^{S/e}_{S, M} @>\sigma>> P \\
@V\psi VV @VV\pi V\\
W_{S/e, N} @>\tilde \sigma>> \GL(V_{S/e}).
\end{CD}
\] 
It is clear that we have $\psi \phi =1$, and hence $\phi$ is injective. 
The proposition is proved. 
\end{proof}

\begin{rem}
\begin{enumerate}
\item When $W$ is a Weyl group, the embedding $\phi$ first appeared in~\cite{Li23}. 

\item Our embedding is not a special case of the embeddings in~\cite{Mu93} and~\cite{EH24}. 
To an edge contraction, there is an obvious surjection $\pi: S\to S/e$. But it is not a Coxeter partition because the condition (2a) in Definition 1.1 in~\cite{EH24}   is not satisfied for $\pi$. 
It is  interesting to provide a unified setting for  the work in {\it loc. cit.} and  ours.  

\item 
Proposition~\ref{W-edge} can be deduced from~\cite[Theorem 3.9]{Dy87} and~\cite{De89} by leveraging the fact that the image of $\phi$ is a reflection subgroup of the Coxeter group $W_S$. 
Our proof is new, as far as we can see. 

\end{enumerate}
\end{rem}

The embedding $\phi$ extends naturally to an embedding 
\[
\phi^\aff: V_{S/e} \rtimes_{\tilde \sigma} W_{S/e, N} \to V_S \rtimes_{\sigma} W_{S, M}, (\alpha, w) \mapsto (\alpha, \phi(w)). 
\]
Let $V_{S, \mbb Z} = \oplus_{s\in \mbb Z} \mbb Z \alpha_s$ be a lattice in $V_{S}$. 
When $W_{S, M}$ is a Weyl group, we can choose $k_{s, s'}$ to be integers  for all $s, s'$ and hence $W_{S, M}$ leaves $V_{S, \mbb Z}$ stable. 
The group $$W_{S, M}^{\aff}= V_{S,\mbb Z} \rtimes_{\sigma} W_{S, M}$$ is the  affine Weyl group associated to $W_{S, M}$. 
Clearly $\phi^{\aff}$ restricts to a group embedding $$\phi^{\aff}: W^{\aff}_{S/e, N} \to W^{\aff}_{S, M}$$ is an affine version of $\phi$.

\subsection{Edge contraction along a linear branch}
\label{linear-branch}

In this section, we assume that $e=\{s_+, s_-\}$ is located on a linear branch, i.e., 
\begin{itemize}
\item[(B)]
{\it there exists a sequence $s_{j_0}=s_-, s_{j_1}=s_+, s_{j_2}, \cdots, s_{j_n}$ such that $m_{s_{j_k}, s_{j_{k+1}}} = 3$ for all $1\leq k\leq n-1$ and $m_{s_{j_k}, s}=2$ for all $k\geq 1$ and  $s \neq s_{j_{k+1}}, s_{j_{k-1}} $.}
\end{itemize}
Pictorially, the graph looks like the following.
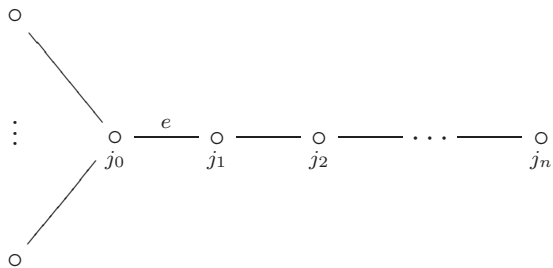
\[
\xymatrix{
\circ \ar@{-}[dr] & & \\
\vdots & \underset{j_0}{\circ} \ar@{-}[r]^e &\underset{j_1}{ \circ}   \ar@{-}[r] & \underset{j_2}{ \circ} \ar@{-}[r] & \ar@{-}[r]   \cdots &\underset{j_n}{ \circ} \\
\circ \ar@{-}[ur] 
}
\]
Let $S'$ be a subset of $S$. Let $W_{S'}$ be the subgroup of $W_S$ generated by $s$ for all $s\in S'$. 
Let $\can: W_{S'} \to W_S$ be the canonical embedding. 

Let $w\in W_S$. Let $\tau_w: W_S\to W_S$ be the conjugation $w' \mapsto ww'w^{-1}$ for all $w'\in W_S$. 

\begin{prop}
\label{tree}
Assuming the condition (B), we have the following commutative diagram. 
\[
\begin{CD}
W_{S-\{s_{j_n}\}} @> \can >> W_S\\
@V\tau VV@VV\tau_{s_{j_1}s_{j_2}\cdots s_{j_n}} V\\
W_{S/e} @>\phi >> W_S
\end{CD}
\]
where $\tau $ is an isomorphism defined  on the generators $s$ by $\tau (s) = s $ if $s\neq s_{j_\ell}$ for $0\leq \ell \leq n$ and 
$$\tau(s_{j_\ell})= 
\begin{cases}
s_{j_{\ell +1}} & \mbox{if}\  \ell \geq 1,\\
s_{0}  & \mbox{if} \ \ell=0.
\end{cases}
$$
\end{prop}

\begin{proof}
For $i\geq 1$, we have 
\begin{align*}
\tau_{s_{j_1}\cdots s_{j_n}} \can (s_{j_\ell} ) 
& = s_{j_1} \cdots s_{j_n} s_{j_\ell} s_{j_n} \cdots s_{j_1}\\
&= 
s_{j_1} \cdots s_{j_{\ell+1}} s_{j_\ell} s_{j_{\ell +1}} \cdots s_{j_1} \\
& = s_{j_1} \cdots s_{j_\ell} ( s_{j_\ell} s_{j_{\ell+1}} s_{j_\ell}) ) s_{j_\ell}\cdots s_{j_1}\\
& = s_{j_1} \cdots s_{j_{\ell-1}} s_{j_{\ell+1}} s_{j_{\ell-1}} \cdots s_{j_1}\\
& =s_{j_{\ell+1}}.
\end{align*}
For $i=1$, we have
\begin{align*}
\tau_{s_1\cdots s_n} \can (s_{j_0} ) 
=s_{j_1} \cdots s_{j_n} s_{j_0} s_{j_n} \cdots s_{j_1}
 = s_{j_1}s_{j_0}s_{j_1}=  s_+ s_-s_+=s_0.
\end{align*}
This implies that the image of $\tau_{s_{j_1}\cdots s_{j_n}} \can$ is the same as the image of $\phi$. And by chasing the generators, we see the well-definedness of the morphism $\tau$. 
This also implies the commutativity of the diagram. By considering the inverse of $\tau_{s_{j_1}\cdots s_{j_n}}$ we see that $\tau$ is isomorphic. This finishes the proof. 
\end{proof}

There exists a canonical embedding $\can: V_{S-\{s_{j_n}\}} \to V_{S}$ and $\can : V_{S/e} \to V_S$. 
Define an linear isomorphism $\sigma^1: V_{S-\{s_{j_n}\} } \to V_{S/e}$ by sending $\alpha_s \to \alpha_s$ for all $s\neq s_{ j_\ell}$ and 
$\alpha_{s_{j_\ell}} \mapsto \sigma_{s_{j_1}} \cdots \sigma_{s_{j_n}} (\alpha_{s_{j_\ell}})$ for all $\ell$. 
The commutative diagram in Proposition~\ref{tree} can lift to a commutative diagram
\[
\begin{CD}
V_{S-\{s_{j_n}\}} \rtimes W_{S-\{j_n\}} @>  \can\times \can  >> V_S\rtimes W_S\\
@V \sigma^1 \times \tau VV @VV \sigma_{s_{j_1}} \cdots \sigma_{s_{j_n}} \times \tau_{s_{j_1} \cdots s_{j_n}} V\\
V_{S/e} \rtimes W_{S/e} @> \phi^\aff>> V_S \rtimes W_S.
\end{CD}
\]

\begin{cor}
Assuming that $W_{S, M}$ is a finite Weyl group so that the assumption (B) holds. We have the following commutative diagram.
\[
\begin{CD}
W^{\aff}_{S-\{s_{j_n}\}} @>  \can\times \can  >> W^{\aff}_S\\
@V \sigma^1 \times \tau VV @VV \sigma_{s_{j_1}} \cdots \sigma_{s_{j_n}} \times \tau_{s_{j_1} \cdots s_{j_n}} V\\
W^{\aff}_{S/e} @> \phi^\aff>> W^{\aff}_S.
\end{CD}
\]
\end{cor}

\subsection{Roots}

In this section, we pick $K$ so that $K_{s, s'} = 2 \cos \frac{\pi}{m_{s,s'}}$. 
The root system of a Coxeter system  $(W_S, S)$ is defined to be $\Phi_S=\{ w (\alpha_s)|  w\in W_S, s\in S\}$. 
Let $\Phi_S^+=\Phi_S \cap \oplus_{s\in S} \mbb Z_{\geq 0} \alpha_s$. It is known that $\Phi_S = \Phi_S^+ \sqcup (-\Phi_S^+)$.
It is clear from the definitions that we have $\Phi_{S/e} \subseteq \Phi_S$ and $\Phi_{S/e}^+\subseteq \Phi_{S}^+$.

\subsection{Length functions}

Let $\ell_S: W_S\to \mbb Z_{\geq 0}$ be the length function by declaring $\ell_S (1)=0$ and $\ell_S(x)$ to be the number of factors in a reduced expression of $x\neq 1$. 
We write $\ell_{S/e}$ for the length function of $W_{S/e}$. 
Write $\mult_{s}(s_1, \cdots, s_n)$ for the multiplicity of $s$ in the tuple $(s_1, \cdots, s_n)$. 

Let $w\in W_{S/e}$. If $w= s_1\cdots s_k$ for $k\in \ell_{S/e} (w)$ is a reduced expression, then $\phi(w) = \phi( s_1) \cdots \phi(s_k)$.
If $s_i=s_0$, then $\phi(s_i) = s_+s_-s_+$ and so we have
\[
\ell_S(\phi(x)) \leq \ell_{S/e}(x) + 2 \ \mult_{s_0} ( s_1, \cdots, s_k). 
\]
If the reduced expression is chosen so that the multiplicity of $s_0$ is the smallest, the above inequality still holds.
The inequality is not sharp, see the following Remark~\ref{length-tree}.  It is interesting to provide a formula for $\ell_S (\phi(x))$. 

\begin{rem}
\label{length-tree}
(1) Assuming the condition (B), we have
\[
\ell_S(\phi (w)) \leq \ell_{S/e} (w)  + 2n, \forall w\neq 1.  
\]

(2) When $W_S=S_{n+1}$ with $S=\{1,\cdots, n+1\}$ and $e=\{i_+, i_++1\}$. 
Regard elements in $W$ as permutations of $\{1,\cdots, n\}$. We have
\[
\ell_S(\phi(w)) =\ell_{S/e} (w) + \# \{ k\in \mbb Z_{\geq 1} | k < i_+, w(k) \geq i_+\} + \# \{ k \in \mbb Z_{\geq 1} | k\geq i_+ , w(k) < i_+\}. 
\]
Since we do not use this formula, we shall skip its proof.

(3) The same formula in (2) applies to the case when $W_S$ is affine type $A$. And one can also deduce a similar formula for all classical Weyl groups.
\end{rem}

\section{Hecke algebras and edge contractions}
\label{Hecke}
Let $v$ be an indeterminate. Let $\A=\mbb Z[v, v^{-1}]$ be the ring of Laurent polynomials with coefficients in $\mbb Z$.

\subsection{Iwahori-Hecke algebras}

Let $W_S$ be the Coxeter group associated to $(S, M)$. 
Let $L_S: S \to \mbb Z$ be a function such that $L_S (s) =L_S(s')$ whenever $s, s' $ are conjugate in $W_S$. 
The function $L_S$ extends to a weight function $L_S: W_S\to \mbb Z$ by imposing the condition
$L_S(ww') = L_S(w) L_S(w')$ if $\ell_S(ww') = \ell_S(w) \ell_S(w')$, where $\ell_S$ is the length function of $W_S$. 
Let $v_s= v^{L_S(s)}$ for all $s\in S$. 

Let $H_S\equiv H_{W, S}$ be the associative algebra over $\A$ with $1$ generated by $T_s$ for $s\in S$ and subject to the following relations.
\begin{align}
\label{H1}
\tag{H1}
& (T_s - v_s) (T_s+v_{s}^{-1}) =0, \ \forall s\in S,\\
\label{H2}
\tag{H2}
& T_s T_{s'} T_s\cdots  = T_{s'} T_s T_{s'} \cdots  , \ \forall s\neq s',  m_{s, s'} \neq \infty,
\end{align}
where there are $m_{s,s'}$ on both sides of the last defining relation. 
$H_S$ is the Hecke, or rather Iwahori-Hecke, algebra associated to the weighted Coxeter systerm $(W, S, L_S)$. 

If $w= s_1\cdots s_k$ is a reduced expression, we write $T_w= T_{s_1} T_{s_2} \cdots T_{s_k}$. 
It is known that $T_w$ is independent of the choice of the reduced expression of $w$. 
It is known that $\{T_w|w\in W_S\}$ is an $\A$-basis of $H_S$. 

Let $(S/e, N)$ be an edge contraction of $(S, M)$ along $e=\{s_{\pm}\}$. 
Since $m_{s_+, s_-}=3$, the weight function $L_S$ induces a weight function $L_{S/e}: S/e\to \mbb Z$ by setting $L_{S/e}(s_0)=L_S(s_+)$
and $L_{S/e}(s)=L_S(s)$ for all $s\neq s_0\in S/e $.
Let $H_{S/e}$ be the Hecke algebra associated to the weighted Coxeter system $(W_{S/e,N}, S/e)$ and $L_{S/e}$.
To avoid confusion, we write $T'_s$ for the generator in $H_{S/e}$ attached to $s\in S/e$.

\begin{prop}
\label{phi-H}
There exists an algebra homomorphism $\phi_v: H_{S/e} \to H_S$ defined by 
$T'_s \mapsto T_s$ for all $s\neq s_0$ and $T'_{s_0} \mapsto T_{s_+ } T_{s_-} T_{s_+}^{-1}$.
\end{prop}

\begin{proof}
We need to show that $\Phi(T'_s)$ for all $s\in S$ satisfies the defining relations for $H_{S/e}$. 
All relations hold naturally except those involve $s_0$. 
For (\ref{H1}), we have
\begin{align*}
(\phi_v (T'_{s_0}) - v_{s_0})  (\phi_v(T'_{s_0}) + v^{-1}_{s_0})
&=(T_{s_+} T_{s_-} T_{s_+}^{-1} - v_{s_-} ) ( T_{s_+} T_{s_-} T_{s_+}^{-1} + v^{-1}_{s_-})\\
&= T_{s_+} ( T_{s_-} - v_{s_-}) ( T_{s_-} + v^{-1}_{s_-}) T_{s_+}^{-1} \\
&=0.
\end{align*}
For (\ref{H2}), if $s'=s_0$ and $m_{s, s_+}=2$, we have
\begin{align*}
\phi_v(T'_s) \phi_ v (T'_{s_0}) \phi_v(T'_s) \cdots 
&= T_{s_+} ( T_{s} T_{s_-} T_{s} \cdots ) T_{s_+}^{-1},\\
\phi_v(T'_{s_0}) \phi_ v (T'_{s}) \phi_v(T'_{s_0}) \cdots
& = T_{s_+}  ( T_{s_-} T_{s} T_{s_-} \cdots ) T_{s_+}^{-1},
\end{align*} 
where the products in the parentheses all have $m_{s, s_-}$ factors.
Hence the condition (\ref{H2}) holds in this case. 
If $s'=s_0$ and $m_{s, s_-}=2$, then the relation (\ref{H2}) remains valid by a similar argument 
thanks to the identity $T_{s_+} T_{s_-} T_{s_+}^{-1} = T_{s_-}^{-1} T_{s_+} T_{s_-}$. 
This finishes the proof. 
\end{proof}

Moreover, we have 

\begin{thm}
\label{H-inj}
The morphism $\phi_v$ is injective. 
\end{thm}

\begin{proof}
Define a ring homomorphism $e_1: \A \to \mbb C$ by sending $v$ to $1$. 
Consider the tensor algebra $\mbb C\otimes_\A H$. 
We write $T_w$ for $1\otimes_\A T_w$. 
It is known that $\mbb C\otimes_\A H \cong \mbb C W, T_w\mapsto w$, the group ring of $W$ over $\mbb C$. 
The morphism $\phi_v$ induces a map $\phi_v|_{v=1}: \mbb C\otimes_\A H_{S/e} \to \mbb C\otimes_\A H_S$. 
From (\ref{H1}), we see that 
\[
T_s^{-1}= T_s + v_s^{-1} - v_s. 
\]
By considering the action of $\phi_v|_{v=1}$ on $T_s'$ for $s\in S/e$, we see that this morphism $\phi_v|_{v=1}$ is 
$\phi$ under the identification $\mbb C\otimes_AH \cong \mbb C W$. 
In particular, we have 
$\phi_v|_{v=1} ( T'_w) = \phi(w)$. 
Now if $x=\sum_{w\in W_{S/e}} c_w T_w' \in \ker \phi_v$, where $c_w\in \A$,  then we must have 
$\phi_v|_{v=1} (x)=0$, which implies that 
\[
\sum_{w\in W_{S/e}} c_w(1)  \phi(w) =0. 
\]
But $\phi$ is injective, this implies that $c_w(1)=0$ for all $w\in W_{S/e}$.
This implies that $c_w/(v-1)\in \A$ for all $w\in W$. 
Assume that not all $c_w$ are zero. 
Now let $n$ be the largest integer such that $c_w/(v-1)^n\in \A$ for all $w$. Fix $w_2$ such that $\frac{c_{w_2}}{(v-1)^n} (1)\neq 0$. Consider
$x/(v-1)^n \in H_{S/e}$. It is clear that $x/(v-1)^n\in \ker \phi_v$ because $(v-1)^n$ is not a zero divisor in $H_{S}$. So we must have 
\[
\sum_{w\in W_{S/e}} \frac{c_w}{(v-1)^n} (1) \phi(w) =0, 
\]
which implies that $ \frac{c_{w_2}}{(v-1)^n} (1)=0$, a contradiction. Therefore $c_w=0$ for all $w\in W_{S/e}$, and $\phi_v$ is injective. 
The theorem is proved. 
\end{proof}

\subsection{Edge contraction along a linear branch, II}
\label{linear-branch-2}

Let $S'\subseteq S$ be a subset. Let $H_{S'}$ be the subalgebra of $H_S$ generated by $s' \in S'$. 
To each $w\in W_S$, we have an automorphism $\tau^v_w: H_S\to H_S $ defined by $T_{w'} \to T_wT_{w'} T^{-1}_w$ for all $w'\in W_s$. 
Let ``$\can$'' be the canonical embedding of $H_{S'}$ to $H_S$. 
There is a similar analogue of Proposition~\ref{tree} for Hecke algebras, whose proof is entirely similar and is hence skipped.

\begin{prop}
\label{tree-2}
Assume that  the condition (B) holds, we have the following commutative diagram. 
\[
\begin{CD}
H_{S-\{s_{j_n}\}} @> \can >> H_S\\
@V\tau^v VV@VV\tau^v_{s_{j_1}s_{j_2}\cdots s_{j_n}} V\\
H_{S/e} @>\phi_v >> H_S
\end{CD}
\]
where $\tau^v $ is an isomorphism defined by $\tau^v (T_s) = T'_{s} $ if $s\neq s_{j_\ell}$ for $0\leq \ell \leq n$ and 
$$\tau^v(T_{s_{j_\ell}})= 
\begin{cases}
T'_{s_{j_{\ell +1}}} & \mbox{if}\  \ell \geq 1,\\
T'_{s_{0}} & \mbox{if} \ \ell=0.
\end{cases}
$$
\end{prop}


\section{Extended affine Hecke algebras of type A}
\label{HA}
In this section, we show that the embedding $\phi_v$ can be generalized to extended affine Hecke algebras   of type $A$. 

\subsection{Iwahori-Matsumoto presentation}

The type-$A$ extended affine Hecke algebra $H^e_r$ of rank $r$ for $r>2$ admits the following Iwahori-Matsumoto presentation.
It is an associative algebra over $\A$ with generators $T_i$ for $i\in \mbb Z/r\mbb Z$ and $T^{\pm 1}_\rho$ subject to the following relations
\[
\begin{cases}
(T_i - v) (T_i + v^{-1})=0\\
T_i T_{i+1} T_i = T_{i+1} T_i T_{i+1} \\
T_i T_j = T_j T_i &\forall i -j \neq 1 (\mbox{mod} \ r)\\
T_\rho T_i = T_{i+1} T_\rho\\
T_\rho T_\rho^{-1} = T_\rho^{-1} T_\rho =1.
\end{cases}
\]
To avoid ambiguity, we put a check  on the generators of $H^e_{r+1}$. We have

\begin{prop}
\label{emb-HA}
Fix $i_-\in \mbb Z/r\mbb Z$. Let $i_+=i_- +1$. 
The following assignment defines an embedding $\phi^e_v: H^e_r\to H^e_{r+1}$. 
\begin{align}
&T_i \mapsto 
\begin{cases}
\check T_i & 1\leq i < i_-, \\
\check T_{i_-} \check T_{i_+} \check T^{-1}_{i_-} & i= i_-\\
\check T_{i+1} & i_+\leq i \leq r, \\
\end{cases}\\
&T_\rho \mapsto \check T^{-1}_{i_+} \check T_\rho.
\end{align}
\end{prop}

\begin{proof}
We must show that the images of the generators of $H^e_r$ under $\phi^e_v$ satisfy the defining relations of $H^e_r$. 
All are known to hold except the fourth one. 
When $i< i_-$, we have
\[
\phi^e_v(T_\rho T_i ) = \phi^e_v (T_{i+1} T_\rho) =
\begin{cases}
\check T_{i+1} \check T^{-1}_{i_+} \check T_\rho & \mbox{if} \ i+1\neq i_-,\\
\check T^{-1}_{i_+} \check T_{i_-} \check T_\rho & \mbox{if} \ i+1 =i_-. 
\end{cases}
\]
When $i=i_-$, we have
\begin{align*}
\phi^e_v(T_\rho T_{i_-} ) 
&= \check T^{-1}_{i_+} \check T_\rho \check T_{i_-} \check T_{i_+} \check T^{-1}_{i_-}\\
&=\check T_\rho \check T_{i_+}\check T^{-1}_{i_-} \\
&= \check T_{i_++1} \check T^{-1}_{i_+} \check T_\rho\\
&= \phi^e_v(T_{i_+} T_\rho).
\end{align*}
When $i_+\leq i\leq r$, we have
\[
\phi^e_v(T_\rho T_i) = \check T^{-1}_{i_+} \check T_{i+2} \check T_\rho = \phi^e_v(T_{i+1} T_\rho). 
\]
Therefore, the fourth defining relation also holds, and so $\phi^e_v$ is an algebra homomorphism. 
Since $T_w T_\rho$ where $w$ runs over the symmetric group of $n$ letters forms a basis for $H^e_v$, we see that $\phi^e_v$ must be injective and thus the proposition is proved. 
\end{proof}

\subsection{Bernstein-Lusztig presentation}

Let $H^{BL}_r$ be an associative algebra over $\A$  defined by the generator-relation presentation  with generators
$T_i$ for all $1\leq i\leq r-1$ and $X_j$ for all $1\leq j\leq r$ and the defining relations.
\begin{align}
\begin{cases}
(T_i - v) (T_i + v^{-1}) =0  & 1\leq i\leq r-1,\\
T_i T_{i+1} T_i = T_{i+1} T_i T_{i+1} & 1\leq i\leq r-2,\\
T_i T_j = T_j T_i & |i-j| \geq 2,\\
T_i X_i T_i = X_{i+1}  & 1\leq  i\leq r-1,\\
T_i X_j = X_j T_i & j \neq i, i+1, 1\leq i\leq r-1. 
\end{cases}
\end{align}
It is well-known that $H^e_r$ is isomorphic to $H^{BL}_r$, and the defining relations of $H^{BL}_r$ is 
the so-called Bernstein-Lusztig presentation of $H^e_r$. An explicit isomorphism $\varphi_r: H^e_r\to H^{BL}_r$ is given by
\begin{align}
\label{IM-BL}
\begin{split}
& T_i \mapsto T_i,  \quad  \forall 1\leq i\leq r-1,\\
& T_\rho\mapsto (T_{r-1} T_{r-2}\cdots T_1X_1)^{-1}.
\end{split}
\end{align}
Note that since $T_r= T_\rho T_{r-1} T^{-1}_\rho$, we see that 
$T_r\mapsto X_1^{-1} T_1^{-1} \cdots T^{-1}_{r-2} T_{r-1} T_{r-2} \cdots T_1$. 
In light of Proposition~\ref{emb-HA}, we have an embedding
$\phi^{BL}_v= \varphi_{r+1} \phi^e_v \varphi^{-1}_r: H^{BL}_r\to H^{BL}_{r+1}$. 

To avoid ambiguity, we place a check on the generators of $H^{BL}_{r+1}$. 

\begin{prop}
\label{emb-BL}
Fix $i_-$ such that $1\leq i_-\leq r-1$. Let $i_+=i_-+1$. 
The embedding $\phi^{BL}_v: H^{BL}_r \to H^{BL}_{r+1}$ is defined by
\begin{align}
\begin{split}
& T_i \mapsto 
\begin{cases}
\check T_i & 1\leq i < i_-,\\
\check T_{i_-} \check T_{i_+}  \check T^{-1}_{i_-} & i=i_-,\\
\check T_{i+1} & i_+\leq i \leq r-1,
\end{cases}\\
&X_j \mapsto 
\begin{cases}
\check X_j + (v-v^{-1}) \check X_{i_+} \check T_j^{-1} \cdots \check T^{-1}_{i_--1} \check T^{-1}_{i_-} \check T^{-1}_{i_--1} \cdots \check T^{-1}_{j} & 1\leq j\leq i_-,\\
\check X_{j+1} & i_+\leq j\leq r. 
\end{cases}
\end{split}
\end{align}
\end{prop}

\begin{proof}
It is straightforward to check that the formula on $T_i$ is correct. 
To easy the burden of notations, we shall drop the check above the generators on $H^e_{r+1}$ and $H^{BL}_{r+1}$ in the proof.
We shall verify the formula on $X_j$ by induction on $j$. 
When $j=1$, we have
\begin{align*}
\phi^{BL}_v(X_1) &= \varphi_{r+1} \phi^e_v ( T^{-1}_1 T^{-1}_2 \cdots T^{-1}_{r-1} T^{-1}_\rho)\\
& = \varphi_{r+1}( T^{-1}_1\cdots T^{-1}_{i_--1} ( T_{i_-} T^{-1}_{i_+} T^{-1}_{i_-} ) T^{-1}_{i_+-1} \cdots T^{-1}_r (T^{-1}_\rho T_{i_+} ))\\
&= T^{-1}_1 \cdots T^{-1}_{i_--1} (T_{i_-} T^{-1}_{i_+} T^{-1}_{i_-} )  (T_{i_+} T_{i_-} T_{i_+}) T_{i_--1} \cdots T_1 X_1\\
&= T^{-1}_1 \cdots T^{-1}_{i_--1} T^2_{i_-} T_{i_--1} \cdots T_1 X_1\\
&= X_1 + (v-v^{-1} )  T^{-1}_1 \cdots T^{-1}_{i_--1} T_{i_-} T_{i_--1} \cdots T_1 X_1\\
&= X_1 + (v-v^{-1} )    T^{-1}_1 \cdots T^{-1}_{i_--1}   X_{i_+} T^{-1}_{i_1} T^{-1}_{i_--1} \cdots T^{-1}_1\\
&= X_1 + (v-v^{-1}) X_{i_+} T^{-1}_1 \cdots T^{-1}_{i_--1}    T^{-1}_{i_1} T^{-1}_{i_--1} \cdots T^{-1}_1 .
\end{align*}
So the formula on $X_j$ is correct when $j=1$. Assume that the formula holds for $j-1<i_-$, we have
\begin{align*}
\phi^{BL}_v(X_j) & = \phi^{BL}_v ( T_{j-1} X_{j-1} T_{j-1}) \\
& = T_{j-1} (X_{j-1} + (v-v^{-1}) X_{i_+} T^{-1} _{j-1}\cdots T^{-1}_{i_--1} T^{-1}_{i_-} T^{-1}_{i_--1} \cdots T^{-1}_{j-1} ) T_{j-1}\\
& = X_j + (v-v^{-1}) X_{i_+} T^{-1} _{j}\cdots T^{-1}_{i_--1} T^{-1}_{i_-} T^{-1}_{i_--1} \cdots T^{-1}_{j}.
\end{align*}
So the formula on $X_j$ holds for all $j\leq i_-$. 
For $j=i_+$, we have
\begin{align*}
\phi^{BL}_v(X_{i_+}) & = \phi^{BL}_v ( T_{i_-} X_{i_-} T_{i_-}) \\
&=(T_{i_-} T_{i_+} T^{-1}_{i_-} ) ( X_{i_-} + (v-v^{-1}) X_{i_+} T^{-1}_{i_-}) (T_{i_-} T_{i_+} T^{-1}_{i_-}) \\
&= T_{i_-} T_{i_+} T_{i_-} X_{i_-} T_{i_-} T_{i_+} T^{-1}_{i_-} - (v-v^{-1}) T_{i_-} T_{i_+} X_{i_-} T_{i_-} T_{i_+} T^{-1}_{i_-} \\
&\hspace{5cm} + (v-v^{-1}) T_{i_-} T_{i_+} T^{-1}_{i_-} X_{i_+} T_{i_+} T^{-1}_{i_-}\\
& = X_{i_++1}.
\end{align*}
So the formula on $X_{i_+}$ is correct. Now, assume that the formula on $X_j$ holds for   all $j-1>i_+$, then we have
\[
\phi^{BL}_v(X_{j}) = \phi^{BL}_v( T_{j-1} X_{j-1} T_{j-1}) = T_j X_j T_j = X_{j+1}.
\]
Therefore, the formula on $X_j$ holds for all $j\geq i_+$. The proposition is thus proved. 
\end{proof}

Note that the evaluation of  $\phi^{BL}_v$ on $X^{-1}_j$ has the following rule
\[
X^{-1}_j \mapsto 
\begin{cases}
X^{-1}_j - (v-v^{-1}) T_j \cdots T_{i_-} \cdots T_j X^{-1}_{i_+} & 1\leq j\leq i_-, \\
X^{-1}_{j+1} & i_+\leq j\leq r.
\end{cases}
\]
By a direction computation, we have

\begin{cor}
\label{emb-BL-c}
The application of $\phi^{BL}_v$ on $X_j X^{-1}_{j+1}$ and $X_j^{-1} X_{j+1}$ is given by
\[
X_j X^{-1}_{j+1} \mapsto 
\begin{cases}
X_j X^{-1}_{j+1} + (v-v^{-1}) T^{-1}_j \cdots T^{-1}_{i_--1} T_{i_-} T_{i_--1} \cdots T_j X_j X^{-1}_{j+1}  &\\
- (v-v^{-1}) T_{j+1} \cdots T_{i_-}\cdots  T_{j+1} X_j X^{-1}_{i_+} & \\
- (v-v^{-1})^2 T_{j+1} \cdots T_{i_-} \cdots T_j X_j X^{-1}_{i_+} & 1\leq j\leq i_--1,\\
X_{i_-} X^{-1}_{i_++1} + (v-v^{-1}) T_{i_-} X_{i_-} X^{-1}_{i_++1} & j= i_-,\\
X_{j+1} X^{-1}_{j+2} & i_+\leq j\leq r-1. 
\end{cases}
\]
\[
X^{-1}_j X_{j+1} \mapsto 
\begin{cases}
X^{-1}_j X_{j+1} - (v-v^{-1}) T_j \cdots T_{i_-} \cdots T_j X^{-1}_{i_+} X_{j+1} &\\
+ (v-v^{-1}) X^{-1}_j X_{i_+} T^{-1}_{j+1} \cdots T^{-1}_{i_-} \cdots T^{-1}_{j+1} &\\
- (v-v^{-1})^2 T^{-1}_{j+1} \cdots T^{-1}_{i_--1} T_{i_-} \cdots T_j & 1\leq j \leq i_--1,\\
X^{-1}_j X_{j+2} - (v-v^{-1}) T_j X^{-1}_{j+1} X_{j+2} & j= i_-,\\
X^{-1}_{j+1} X_{j+2} & i_+\leq j\leq r-1. 
\end{cases}
\]
\end{cor}

\section{Affine Hecke algebras}
\label{AffH}
In this section, we present an embedding of affine Hecke algebras along a linear tree. 
We fix the matrix $K=(K_{s, s'})$ where $K_{s, s'} = 2\cos \frac{\pi}{m_{s, s'}}$ in this section.
We further assume that $-K$ is a Cartan matrix so that the group $W_S$ is a Weyl group. 
  
\subsection{Bernstein-Lusztig presentation}

In this section, 
we recall the affine Hecke algebras $H^{\aff}_S$ associated with an adjoint root system and equal parameter following~\cite{L89}.
As an associative algebra over $\A$, $H^{\aff}_S$ admits a generator-relation presentation as follows.
\begin{enumerate}
\item Generators: $T_s$, $\theta_s^{\pm 1}$ for $s\in S$;
\item Relations: $T_s$ satisfies the relations (\ref{H1}) and (\ref{H2});
\begin{align}
\theta_s \theta_s^{-1} & = \theta_{s}^{-1}\theta_s=1,\tag{H3}\\
\theta_s\theta_{s'} & = \theta_{s'} \theta_s,\tag{H4} \\
\theta_s T_{s'} & = T_{s'} \theta_{x_{s, s'}}  + (v- v^{-1}) \frac{\theta_s - \theta_{x_{s, s'}}}{1-\theta_{- s'}}, \tag{H5} 
\end{align}
where $x_{s, s'}=s+k_{s,s'}s'  $ and
 $\theta_{x_{s, s'}}= \theta_s \theta^{k_{s,s'}}_{s'}$. 
\end{enumerate}

The algebra $H^{\aff}_S$ is a quantum version of the group $W^{\aff}_{S, M}$. 
The generator $\theta_s$ is the generator attached to the simple root $\alpha_s$. 

\subsection{The embedding}
Assume that the condition (B) holds. 
Due to our assumption in this section, such a  condition can always be achieved. 
We have

\begin{prop}
Retaining the assumption in this section. 
We have the following embedding
\[
\phi^{\aff}_v: H^{\aff}_{S/e} \to H^{\aff}_S.
\]
defined by
\begin{align}
T_s &\mapsto 
\begin{cases}
T_{s_+} T_{s_-} T^{-1}_{s_+}, & s= s_0\\
T_s& s\neq s_0,
\end{cases}\\
\theta_{s} & \mapsto 
\begin{cases}
\theta_s - (v-v^{-1}) T_{s_{j_\ell}} \cdots T_{s_{j_1}} \cdots T_{s_{j_\ell}}   \theta^{-1}_{s_{j_1}} \cdots\theta^{-1}_{s_{j_{\ell-1}}}&\\
+(v-v^{-1}) \theta_{s_{j_\ell}} \cdots \theta_{s_{j_1}} T^{-1}_{s_{j_{\ell-1}}} \cdots 
T^{-1}_{s_{j_{1}}} \cdots T^{-1}_{s_{j_{\ell -1}}} &\\
- (v-v^{-1})^2 T^{-1}_{s_{j_{\ell-1}}} \cdots T^{-1}_{s_{j_2}} \cdots T_{s_{j_1}} \cdots T_{s_{j_\ell}}, 
& s=s_{j_\ell}, \ell \geq 2,\\
\theta_{s_+} \theta_{s_-} - (v-v^{-1}) T_{s_+} \theta_{s_-} & s=s_0, \\
\theta_s, &\mbox{o.w.}
\end{cases}\\
\theta^{-1}_s & \mapsto
\begin{cases}
\theta^{-1}_{s} + (v-v^{-1}) T^{-1}_{s_{j_\ell}} \cdots T^{-1}_{s_{j_2}} T_{s_{j_1}} \cdots T_{s_{j_\ell}} \theta^{-1}_{s_{j_\ell}}\\
-(v-v^{-1}) T_{s_{j_{\ell-1}}}  \cdots T_{s_{j_1}} \cdots T_{s_{j_{\ell-1}}} \theta^{-1}_{s_{j_{\ell}}}\cdots \theta^{-1}_{s_{j_1}} \\
-(v-v^{-1})^2 T_{s_{j_{\ell-1}}} \cdots T_{s_{j_1}} \cdots T_{s_{j_\ell}} \theta^{-1}_{s_{j_\ell}} \cdots
\theta^{-1}_{j_1}, 
& s=s_{j_\ell}, \ell \geq 2,\\
\theta^{-1}_{s_+} \theta^{-1}_{s_-} + (v-v^{-1}) T_{s_{i_-}} \theta^{-1}_{s_+} \theta^{-1}_{s_-},
& s= s_0,\\
\theta^{-1}_s, & \mbox{o.w.}
\end{cases}
\end{align}
\end{prop}

\begin{proof}
The defining relations are local and reduced to the affine type A case, which follows from Corollary~\ref{emb-BL-c}. 
\end{proof}

Let $(S^\aff, M^\aff)$ be affine Coxeter system associated with $(S, M)$. 
There is an isomorphism $H_{S^\aff}$ with $H^\aff_S$. It is expected that the embedding
$\phi_v$ on $H_{S^\aff/e}$ gets identified with $\phi^\aff_v$ on $H^\aff_{S/e}$ under the isomorphism $H_{S^\aff}\cong H^\aff_S$. 

Note that the algebra $H^{BL}_r$ is a quantization of the extended affine Weyl group of type A, i.e.,
$\mbb Z^n \rtimes S_n $. A generalization of $\phi^{BL}_v$ in Proposition~\ref{emb-BL} to other cases remains to be done.

\section{Schur-Jimbo duality and edge contractions}
\label{Schur}
In this section, we study the compatibility of the embeddings of Hecke algebras of type $A$ with the embeddings of quantum $\mathfrak{sl}_n$ in~\cite{Li22} (see also~\cite{Li23}). 

\subsection{Compatibility}

Let $\U_n$ be an associative algebra over $\mbb Q(v)$ defined by the generator-relation presentation.
\begin{enumerate}
\item Generators: $E_i, F_i, K_i^{\pm 1}$ for all $1\leq i\leq n-1$;

\item Relations:  for all $1\leq i, j\leq n$,
\begin{align*}
\begin{cases}
K_i K_j = K_j K_i, 
K_i K_i^{-1} =K_i^{-1} K_i=1,\\
K_i E_j = v^{2\delta_{ij} - \delta_{i, j+1} -\delta_{i, j-1} }E_j K_i,\\
K_i F_j = v^{-2\delta_{ij} + \delta_{i, j+1} + \delta_{i, j-1}} F_j K_i,\\
 E_i F_j - F_j E_i =\delta_{ij} \frac{K_i - K^{-1}_i}{v- v^{-1}}, \\
 E^2_i E_j - (v+v^{-1}) E_i E_j E_i + E_j E_i^2=0,  & \forall |i-j|=1,\\
E_iE_j - E_j E_i=0, & \forall |i-j|>1,\\
F^2_i F_j - (v+v^{-1}) F_i F_j F_i + F_j F_i^2=0,  &  \forall |i-j|=1,\\
F_iF_j-F_j F_i =0, &  \forall |i-j|>1,
\end{cases}
\end{align*}
\end{enumerate}
where $\delta_{ij}$ is the Kronecker delta. 
The algebra $\U_n$ is a quantum $\mathfrak{sl}_n$ and it is a Hopf algebra whose comultiplication is given by
\[
\Delta(E_i) = E_i \otimes K_i + 1\otimes E_i, \Delta(F_i) = F_i\otimes 1 + K^{-1}_i \otimes F_i, \Delta(K_i)=K_i \otimes K_i, 
\]
for all $1\leq i\leq n-1$. 
Let $\V_n$ be an $n$-dimensional vector space over $\mbb Q(v)$ with a fixed basis $\{ e_i| 1\leq i\leq n\}$. 
The natural $\U_n$-representation on $\V_n$ is given by 
\[
E_i .e_j = \delta_{i+1, j} e_i,
F_i .e_j =\delta_{i, j} e_{i+1},
K_i .e_j =v^{\delta_{i, j} - \delta_{i+1, j}} e_j,   
\]
for all $1\leq i\leq n-1$, $1\leq j\leq n$. 
Let $\T_{n, d}$ be the tensor space $\V_n^{\otimes d}$. 
The tensor space $\T_{n, d}$ carries a left $\U_n$-action via the comultiplication $\Delta$. 
We write $[r_1,\cdots, r_d]$ for the basis element $e_{r_1}\otimes \cdots \otimes e_{r_d}$ in $\T_{n, d}$. 
More precisely, the actions can be described explicitly as follows.
For all $1\leq i\leq n-1$, and $1\leq r_j\leq n$ for all $1\leq j\leq d$, we have
\begin{align}
\label{U-action}
\begin{cases}
E_i . [r_1,\cdots, r_d]  =
\sum_{1\leq p\leq d: r_p=i+1} v^{\#\{ j>p| r_j =i\} - \# \{ j>p| r_j = i+1\}} [ \cdots r_{p-1}, r_{p} -1, r_{p+1} \cdots],\\
F_i. [r_1,\cdots, r_d]  = 
\sum_{1\leq p\leq d: r_p = i} v^{\#\{ j< p| r_j =i+1\} - \# \{ j<p|r_j=i\}} [\cdots r_{p-1}, r_p +1, r_{p+1}\cdots],\\
K_i .[r_1,\cdots, r_d]   = v^{\# \{ j| r_j= i\}- \#\{j| r_j= i+1\}} [r_1,\cdots, r_d]. 
\end{cases}
\end{align}

Consider the  Hecke algebra $H^{A_{d-1}}_S$ of type $A_{d-1}$ where $S=\{1, 2, \cdots, n-1\}$ and the Coxeter matrix $M$ is given by 
$m_{s,s}=1$, $m_{s, s'}= 3 $ if $|s-s'|=1$ and $m_{s, s'}=0$, otherwise.
Let $\H_d= \mbb Q(v)\otimes_\A H^{A_{d-1}}_S$.   
The tensor space $\T_{n,d}$ carries a right $\H_d$-action as follows. For all $1\leq i\leq n-1$
\begin{align}
\label{H-action}
[r_1,\cdots, r_d] . T_i = 
\begin{cases}
[\cdots r_{i+1}, r_i \cdots], & \mbox{if}\ r_i < r_{i+1},\\
v[\cdots r_i r_{i+1} \cdots], & \mbox{if} \ r_i=r_{i+1},\\
(v-v^{-1}) [\cdots r_i r_{i+1}\cdots] + [\cdots r_{i+1} r_i\cdots], &\mbox{if} \ r_i> r_{i+1}.
\end{cases}
\end{align}
It is known that the $\U_n$-action and $\H_d$-action on $\T_{n, d}$ commute and hence $\T_{n,d}$ carries a $\U_n-\H_d$-bimodule structure. Moreover, the triple $(\U_n, \T_{n, d}, \H_d)$ satisfies the double centralizer property, i.e., the Schur-Jimbo duality. 

Now fix a pair $(i_+, i_-)$ such that $i_-=i_+ + 1$ and $1\leq i_+< i_-\leq n-1$. 
Let $\phi_v: \H_d\to \H_{d+1}$ be the embedding associated to $(i_+, i_-)$ defined in Proposition~\ref{phi-H}. 
Fix $\ve\in \{\pm 1\}$. Let $\phi_\ve: \U_n\to \U_{n+1} $ be the finite analogue of the embedding $\Phi_{r, \ve}$ in~\cite{Li22}. In particular, it is defined by the following assignment.
\begin{align*}
& E_i \mapsto 
\begin{cases}
E_i & \mbox{if} \ 1\leq i< i_+, \\
E_i E_{i+1} - v^{\ve} E_{i+1} E_{i} & \mbox{if} \ i= i_+,\\
E_{i+1} & \mbox{if} \ i> i_+,
\end{cases}\\
& F_i \mapsto 
\begin{cases}
F_i & \mbox{if} \ 1\leq i< i_+,\\
F_{i+1} F_i - v^{-\ve} F_i F_{i+1} & \mbox{if} \ i = i_+, \\
F_{i+1} & \mbox{if} \ i> i_+,
\end{cases}\\
& K_i \mapsto
\begin{cases}
K_i & \mbox{if} \ 1\leq i< i_+, \\
K_i K_{i+1} & \mbox{if} \ i = i_+,\\
K_{i+1} & \mbox{if} \ i> i_+. 
\end{cases}
\end{align*}

Let $\phi^\T: \T_{n, d}\to \T_{n+1, d+1}$ be a $\mbb Q(v)$-linear map defined by  the rule
\[
[r_1,\cdots, r_d] \mapsto [\breve r_1, \cdots \breve r_{i_+}, i_-, \breve r_{i_-} \cdots, \breve r_d] 
\]
where 
$$\breve r_i =\begin{cases}
r_i, &\mbox{if}\ r_i\leq i_+\\
r_i+1 &\mbox{if} \ r_i > i_+
\end{cases}
$$
\begin{prop}
The map $\phi^\T$ is compatible with the bimodule structures. 
More precisely, we have the following commutative diagram.
\[
\begin{CD}
\U_n\times \T_{n,d} @>>> \T_{n, d} \\
@V\phi_\ve\times \phi^\T VV @VV\phi^\T V  \\
\U_{n+1} \times \T_{n+1, d+1} @>>> \T_{n+1, d+1} 
\end{CD}
\]
where the horizontal maps are respective actions. 
\end{prop}

\begin{proof}
This has been observed in the proof of Theorem 2.1.1 in~\cite{Li22}. 
The main ingredient in that proof is that $E_{i_-}E_{i_+}$ and $F_{i_+} F_{i_-}$ act trivially on $\V_{n+1}$.
\end{proof}

Define a map $\phi^\T_1: \T_{n, d} \to \T_{n+1, d+1}$ be the $\mbb Q(v)$-linear map defined by the rule
\[
[\breve r_1,\cdots, \breve r_d] \mapsto [\breve r_1, \cdots, \breve r_{i_+}, n+1, \breve r_{i_-}, \cdots, \breve r_d]. 
\]

\begin{prop}
We have the following commutative diagram.
\[
\begin{CD}
\T_{n, d} @<<< \T_{n, d} \times \H_d\\
@V\phi^\T_1 VV @VV\phi^\T_1\times \phi_v V \\
\T_{n+1, d+1} @<<< \T_{n+1, d+1} \times \H_{d+1}.
\end{CD}
\]
\end{prop}

\begin{proof}
It is enough to show that it is commutative when applying the generator $T_{i_0}$. But this can be checked by a direct computation as follows. 
For $1\leq r_{i_+} < r_{i_-} \leq n$, we have
\begin{align*}
\begin{split}
& [\cdots r_{i_+}, n+1, r_{i_-}\cdots] T_{i_+} T_{i_-} T_{i_+} = (v-v^{-1}) [\cdots n+1, r_{i_-}, r_{i_+}\cdots] + [\cdots r_{i_-} , n+1, r_{i_+}\cdots], \\
& [\cdots r_{i_+}, n+1, r_{i_-}\cdots] T_{i_+} T_{i_-}  = (v-v)^{-1} [\cdots n+1, r_{i_-} , r_{i_+} \cdots].
\end{split}
\end{align*}
This yields that 
\begin{align}
\label{comm-1}
[\cdots r_{i_+}, n+1, r_{i_-}\cdots] T_{i_+} T_{i_-} T_{i_+}^{-1} = [\cdots r_{i_-} , n+1, r_{i_+}\cdots],\ \mbox{if} \ r_{i_+}< r_{i_-}. 
\end{align}
If $r_{i_+} =r_{i_-}$, we have
\begin{align*}
\begin{split}
& [\cdots r_{i_+}, n+1, r_{i_-}\cdots] T_{i_+} T_{i_-} T_{i_+} = v(v-v^{-1}) [\cdots n+1, r_{i_-} , r_{i_+} \cdots] + v[\cdots r_{i_-} ,n+1, r_{i_+}\cdots],\\
&  [\cdots r_{i_+}, n+1, r_{i_-}\cdots] T_{i_+} T_{i_-}= v[\cdots n+1, r_{i_-} r_{i_+}\cdots] . 
\end{split}
\end{align*}
So we have 
\begin{align}
\label{comm-2}
[\cdots r_{i_+}, n+1, r_{i_-}\cdots] T_{i_+} T_{i_-} T_{i_+}^{-1} = v[\cdots r_{i_-}, n+1, r_{i_+}\cdots], \ \mbox{if} \ r_{i_+}=r_{i_-}. 
\end{align}
If $r_{i_+} > r_{i_-}$, we have 
\begin{align*}
 [\cdots r_{i_+}, n+1, r_{i_-}\cdots] T_{i_+} T_{i_-} T_{i_+} = (v-v^{-1})^2 [\cdots n+1, r_{i_+} , r_{i_-} \cdots] \\
 + (v-v^{-1}) [\cdots r_{i_+}, n+1, r_{i_-} \cdots]\\
+  (v-v)^{-1} [\cdots n+1, r_{i_-} , r_{i_+}\cdots] + [\cdots r_{i_-} , n+1, r_{i_+}\cdots],\\
 [\cdots r_{i_+}, n+1, r_{i_-}\cdots] T_{i_+} T_{i_-} = (v-v^{-1}) [\cdots n+1, r_{i_+} , r_{i_-}\cdots] + [\cdots n+1, r_{i_-}, r_{i_+}\cdots]. 
\end{align*}
Then we have
\begin{align}
\label{comm-3}
\begin{split}
[\cdots r_{i_+}, n+1, r_{i_-}\cdots] T_{i_+} T_{i_-} T_{i_+}^{-1}= 
(v-v^{-1}) [\cdots r_{i_+}, n+1, r_{i_-}\cdots]\\ + [\cdots r_{i_-} , n+1, r_{i_+}\cdots], \mbox{if} \ r_{i_+}> r_{i_-}.
\end{split}
\end{align}
By (\ref{comm-1})-(\ref{comm-3}), we see that the right square in the above diagram commutes when applies $T_{i_0}$. This finishes the proof. 
\end{proof}

\begin{rem}
Recently Ueda established an embedding from affine Yangian of $\mathfrak{sl}_n$ to its higher rank in~\cite{U23}. The embedding is shown to be compatible with
the degenerate affine Schur-Weyl duality~\cite{U24}. It is interesting to compare his work with us. 
\end{rem}

\section{Appendix: Edge contraction along a labelled edge}

In this appendix, we study the effect of edge contraction on a labelled edge in a Coxeter graph.
This study is suggested by the anonymous referee. 

\subsection{}
Recall the  Coxeter system $(W_S, S)$ associated with the Coxeter matrix $M$ from Section~\ref{Coxeter}. 
Let $K=(K_{s, s'})_{s, s'\in S}$ be the matrix defined by
$K_{s, s'} = 2 \cos \frac{\pi}{m_{s, s'}}$. 
In this appendix, we assume that the edge $e=\{s_+, s_-\}$ in the Coxeter graph is a labelled edge, i.e.,
\begin{align}
\label{lab}
m_{s_+, s_-} \geq 4. 
\end{align}
We can still consider the subgroup of $W_S$ generated by the reflections
$s\in S-\{s_{\pm}\}$ and $s_+s_-s_+$. By the works~\cite{Dy87},~\cite{De89} of Dyer and Deodhar, we know that it is a Coxeter group.
In this section, we shall determine its type, which can be done by the following operation on the original Coxeter matrix. 
Note that there is an algorithm in~\cite{Dy87} and~\cite{De89} in determining the Coxeter type as well.

Recall $S/e= S-\{ s_+, s_-\} +\{s_0\}$. 
We define a Coxeter matrix $\tilde N= (\tilde n_{s, s'})_{s, s' \in S/e}$ by the following rule. For any $s\neq s'$, we
define
\begin{align}
\tilde n_{s, s'} =
\begin{cases}
m_{s, s'} &\mbox{if} \ s, s'\neq s_0,\\
m_{s, s_-} & \mbox{if}\ s'=s_0, m_{s, s_+} = 2,\\
m_{s_+, s_-} & \mbox{if} \ s'=s_0, m_{s, s_+} =3, m_{s, s_-}=2,\\
\infty &\mbox{if} \ s'=s_0, m_{s,s_+} =3, m_{s, s_-}\geq 3,\\
\infty &\mbox{if}\ s'=s_0, m_{s, s_+} \geq 4.
\end{cases}
\end{align}
The undefined entries in $\tilde N$ can be determined by the fact that $\tilde N$ is symmetric.

Here is an example in terms of Coxeter graphs.
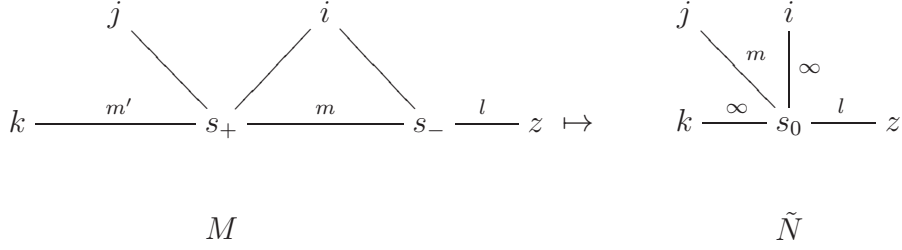
\[
\xymatrix{
&j \ar@{-}[dr] && i \ar@{-}[dl] \ar@{-}[dr] &&\\
k \ar@{-}[rr]^{m'} && s_+\ar@{-}[rr]^m & &s_- \ar@{-}[r]^l & z\\
&&M&&
}
\xymatrix{
& j \ar@{-}[dr]^m& i \ar@{-}[d]^\infty&\\
\mapsto& k \ar@{-}[r]^\infty&s_0\ar@{-}[r]^l &z\\
&&\tilde N&
}
\]
In light of this example, we call the Coxeter matrix $\tilde N$ the edge contraction of $M$ along the arrow $\overrightarrow {e}=s_+\to s_-$. 

Let  $(W_{S/e, \tilde N}, S/e)$ be the Coxeter system associated with $\tilde N$. The following is 
the counterpart of Proposition~\ref{W-edge}.

\begin{prop}
\label{W-edge-2}
There is a group embedding $\tilde \phi: W_{S/e, \tilde N} \to W_{S, M}$ defined by 
\begin{align}
\tilde \phi (s) =
\begin{cases}
s &\mbox{if} \ s\neq s_0,\\
s_+ s_- s_+ & \mbox{if} \ s= s_0.
\end{cases}
\end{align}
\end{prop}

\begin{proof}
The fact that $\tilde \phi$ is a group homomorphism can be checked in a straightforward way.
Actually, by definition, only the following two cases are nontrivial to check.
The first case is when $s'=s_0$ and $m_{s, s_+}=2$. In this case, we have the following computation.
\begin{align}
\begin{split}
(\tilde \phi(s) \tilde \phi(s'))^{n_{s,s'}} & = 
(s (s_+s_-s_+))^{n_{s,s'}}\\
&= (s_+ ss_- s_+)^{m_{s, s_-}}\\
& = s_+ (ss_-)^{m_{s,s_-}}s_+\\
& = 1
\end{split}
\end{align}
The second case is when $s'=s_0$, $m_{s, s_+}=3$, $m_{s, s_-}=2$.
In the case, we have the following computation.
\begin{align}
\begin{split}
(\tilde \phi(s) \tilde \phi(s'))^{n_{s,s'}} & = 
(s (s_+s_-s_+))^{n_{s,s'}}\\
&=(s s_+s_-s_+) (s s_+s_-s_+) \cdots (s s_+s_-s_+)  \hspace{2.3cm} \mbox{($m_{s_+,s_-}$ parentheses)}\\
&=s_+  (s s_+s) s_- (ss_+s) s_- ( ss_+s) \cdots (ss_+s) s_- s_+ \hspace{1cm} \mbox{($m_{s_+,s_-}$ parentheses)}\\
&=s_+ s (s_+s_-)^{m_{s_+,s_-}} s s_+\\
& =1
\end{split}
\end{align}
This shows that $\tilde \phi$ is a group homomorphism.

It remains to prove the injectivity of $\tilde \phi$. The proof goes in a similar manner as that of Proposition~\ref{W-edge}. 
We shall highlight the places where we need to modify. 
Define the symmetric matrix $\hat K=(\hat K_{s,s'})_{s, s'\in S/e}$ by
\begin{align}
\hat K_{s, s'}
=
\begin{cases}
K_{s, s'} &\mbox{if} \ s, s' \neq s_0,\\
K_{s, s_-} + K_{s, s_+} K_{s_+, s_-} &\mbox{if} \ s=s_0, s'\neq s_0. 
\end{cases}
\end{align}
Then $\hat K_{s,s'}$ is subject to the conditions in (\ref{K}) for $\tilde N$. 

Indeed, when $s'=s_0$ and $m_{s, s_+}=2$, we have $K_{s, s_+}=0$ and 
\[
\hat K_{s, s'}= K_{s, s_-}.
\]
When $s'=s_0$ and $m_{s, s_+}=3$ and $m_{s, s_-}=2$, we have $K_{s,s_-}=0$ and $K_{s, s_+}=1$ and so
\[
\hat K_{s, s'} = K_{s_+,s_-}.
\]
When $s'=s_0$, $m_{s, s_+}=3$ and $m_{s, s_-}\geq 3$, we have
\[
\hat K_{s, s'} = K_{s, s_-} +K_{s, s_+} K_{s_+,s_-} =
K_{s, s_-} + K_{s_+,s_-} \geq 2. 
\]
When $s'=s_0$ and $m_{s, s_+}\geq 4$, we have
\[
\hat K_{s, s'} = K_{s, s_-} + K_{s, s_+} K_{s_+,s_-}\geq
K_{s, s_+} K_{s_+,s_-} \geq 4(\cos \frac{\pi}{4} )^2= 2. 
\]
With the above information, one can verify readily that $\hat K$ satisfies (\ref{K}) for $\tilde N$. 
As such, this defines an embedding $\tilde \sigma: W_{S/e,\tilde N} \to \GL(V_{S/e})$ using the techniques in Section~\ref{linear}. 

We define an injective linear map
\[
\iota: V_{S/e} \to V_{S}, \alpha_s\mapsto \alpha_s, \forall s\neq s_0, \alpha_{s_0} \mapsto \alpha_{s_-}+ K_{s_+,s_-}\alpha_{s_+}=\sigma_{s_+}(\alpha_{s_-}).
\]
We identify $V_{S/e}$ with its image under $\iota$.
Then one can check that the elements $\sigma_s$, for all $s\neq s_{\pm}$ and $\sigma_{s_+} \sigma_{s_-}\sigma_{s_+}$ leave
$V_{S/e}$ stable. 

Indeed, we have for all $s\in S$,
\begin{align}
\label{comp-a}
\begin{split}
\sigma_{s_+}\sigma_{s_-}\sigma_{s_+} (\alpha_s)
& = \sigma_{s_+}\sigma_{s_-} (\alpha_s + K_{s_+,s}\alpha_{s_+})\\
& = \sigma_{s_+} ( \alpha_s + K_{s_-,s} \alpha_{s_-} + K_{s_+,s} (\alpha_{s_+} + K_{s_-, s_+} \alpha_{s_-}))\\
& =\alpha_s + K_{s_+, s} \alpha_{s_+} + K_{s_-, s} ( \alpha_{s_-}+K_{s_+ , s_-} \alpha_{s_+}) +\\
& \hspace{1.5cm} K_{s_+,s} ( - \alpha_{s_+} + K_{s_-, s_+} (\alpha_{s_-} +K_{s_+, s_-} \alpha_{s_+}))\\
& =\alpha_s + \hat K_{s_0, s} \sigma_{s_+}(\alpha_{s_-})\\
&= \alpha_s + \hat K_{s_0,s} \iota(\alpha_{s_0}).
\end{split}
\end{align}
And we have for all $s\neq s_{\pm}$, 
\begin{align}
\label{comp-b}
\begin{split}
\sigma_s (\iota(\alpha_{s_0})) & = \sigma_s (\alpha_{s_-}+K_{s_+, s_-} \alpha_{s_+})\\
&= \alpha_{s_-} + K_{s, s_-}\alpha_s + K_{s_+,s_-} (\alpha_{s_+} + K_{s, s_+} \alpha_s) \\
&=\iota(\alpha_{s_0}) + \hat K_{s,s_0} \alpha_s. 
\end{split}
\end{align}

Moreover, by (\ref{comp-a}) and (\ref{comp-b}), we know that when restrict to
the subspace $V_{S/e}$, the operators $\sigma_s$, $s\in S/e-\{s_0\}$ and $\sigma_{s_+}\sigma_{s_-}\sigma_{s_+}$ coincide
respectively with the operators $\tilde \sigma_s$, $s\in S/e-\{s_0\}$ and $\tilde \sigma_{s_0}$ on $V_{S/e}$.
Now the remaining arguments in the proof of Proposition~\ref{W-edge} apply here to infer that $\tilde \phi$ is injective.
The proposition is thus proved. 
\end{proof}

Let $H_{S/e, \tilde N}$ be the Hecke algebra associated with the Coxeter system $(W_{S/e, \tilde N}, S/e)$. We have
the following counterpart of Theorem~\ref{H-inj}.

\begin{prop}
\label{phi-H-ext}
There is an algebra embedding $\tilde \phi_v: H_{S/e, \tilde N} \to H_{S}$ defined by
\begin{align}
T'_s \mapsto 
\begin{cases}
T_s &\mbox{if}\ s\in S/e-\{s_0\},\\
T_{s_+} T_{s_-} T^{-1}_{s_+} &\mbox{if} \ s=s_0.
\end{cases}
\end{align}
\end{prop}

\begin{proof}
One needs to show that $\tilde \phi_v$ is well-defined. These can be checked by definition. 
It is easy to see that we only need to verify the braid relations involving the generators $T'_{s_0}$. 
The case when $s'=s_0$ and $m_{s_+, s} =2$ follows readily due to the commutator relation
$T_{s} T_{s_+}=T_{s_+} T_s$. 
The remaining case is 
 $s'=s_0$, $m_{s, s_+}=3$, $m_{s, s_-}=2$ and $m_{s_+, s_-}<\infty$. In this case, we have
\[
T^{-1}_{s_+} T_s T_{s_+} = T_s T_{s_+} T^{-1}_s\ \mbox{and}\ 
T^{-1}_s T^{-1}_{s_+} T_s = T_{s_+} T^{-1}_s T^{-1}_{s_+}.
\]
So if $m_{s_+, s_-}$ is even, we have 
\begin{align}
\label{H-a}
\begin{split}
& T_s (T_{s_+} T_{s_-} T^{-1}_{s_+}) T_s ( T_{s_+} T_{s_-} T^{-1}_{s_+} ) \cdots T_s (T_{s_+} T_{s_-} T^{-1}_{s_+}) \\
&= T_{s_+} ( T^{-1}_{s_+} T_s T_{s_+}) T_{s_-} ( T^{-1}_{s_+} T_s T_{s_+}) T_{s_-} \cdots (T^{-1}_{s_+} T_s T_{s_+} ) T_{s_-} T^{-1}_{s_+} \\
& = T_{s_+} ( T_s T_{s_+} T^{-1}_s) T_{s_-}  ( T_s T_{s_+} T^{-1}_s) \cdots (T_s T_{s_+} T^{-1}_s) T_{s_-} T^{-1}_{s_+}\\
&= T_{s_+} T_s ( T_{s_+} T_{s_-} \cdots T_{s_+} T_{s_-}) T^{-1}_s T^{-1}_{s_+},
\end{split}
\end{align}
where the monomial at the beginning has $m_{s_+, s_-}$ terms with $(T_{s_+} T_{s_-} T^{-1}_{s_+})$ counted as one term, and as such
there are $m_{s_+, s_-}$ terms in the parenthesis in the end.
If $m_{s_+, s_-}$ is odd, then 
\begin{align}
\label{H-b}
\begin{split}
 &T_s (T_{s_+} T_{s_-} T^{-1}_{s_+}) T_s ( T_{s_+} T_{s_-} T^{-1}_{s_+} ) \cdots ( T_{s_+} T_{s_-} T^{-1}_{s_+} ) T_s\\
 & = T_{s_+} (T^{-1}_{s_+} T_s T_{s_+}) T_{s_-}  (T^{-1}_{s_+} T_s T_{s_+}) T_{s_-}   \cdots (T^{-1}_{s_+} T_s T_{s_+}) T_{s_-}  
 T^{-1}_{s_+} T_s\\
 & =T_{s_+} ( T_s T_{s_+} T^{-1}_s) T_{s_-}   ( T_s T_{s_+} T^{-1}_s) T_{s_-} \cdots ( T_s T_{s_+} T^{-1}_s) T_{s_-}  T^{-1}_{s_+} T_s\\
 &= T_{s_+} T_s ( T_{s_+} T_{s_-} T_{s_+} T_{s_-} \cdots T_{s_+} T_{s_-}) T^{-1}_s T^{-1}_{s_+} T_s\\
 &= T_{s_+} T_s ( T_{s_+} T_{s_-} T_{s_+} T_{s_-} \cdots T_{s_+} T_{s_-})  T_{s_+} T^{-1}_s T^{-1}_{s_+}\\
 &=  T_{s_+} T_s ( T_{s_+} T_{s_-} \cdots T_{s_+} T_{s_-}T_{s_+}) T^{-1}_s T^{-1}_{s_+}.
\end{split}
\end{align}
Similarly, one can check that we have
\begin{align}
\label{H-c}
\begin{split}
(T_{s_+} T_{s_-} T^{-1}_{s_+}) T_s (T_{s_+} T_{s_-} T^{-1}_{s_+}) T_s \cdots 
& = T_{s_+} T_s ( T_{s_-} T_{s_+} T_{s_-} T_{s_+} \cdots) T^{-1}_s T^{-1}_{s_+} ,
\end{split}
\end{align}
where the left hand side has $m_{s_+, s_-}$ terms and hence there are $m_{s_+,s_-}$ terms in the parenthesis of the right hand side. 
So (\ref{H-c}) is equal to (\ref{H-a}) if $m_{s_+, s_-}$ is even and (\ref{H-b}) if $m_{s_+, s_-}$ is odd.
Thus we see that $T_s$ and $T_{s_+} T_{s_-}T_{s_+}$ subject to the braid relation of 
$T'_s$ and $T'_{s_0}$ in this case.  Therefore the morphism $\tilde \phi_v$ is well-defined.

The injectivity of $\tilde \phi_v$  can be proved by the same argument
in the proof of Theorem~\ref{H-inj}. This finishes the proof. 
\end{proof}

We end the appendix with the following question, generalizing Proposition~\ref{phi-H-ext}. 
Let $W'$ be a reflection subgroup of $W_S$. 
By~\cite{Dy87} and~\cite{De89}, there is a set $S'$ of reflections in $W_S$ such that 
the pair $(W', S')$ is a Coxeter system.
Assume that  $S' =\{ w_i s_i w^{-1}_i| s_i \in S, w_i\in W\}$. 
Let $H_{W', S'}$ be the Hecke algebra associated with the Coxeter system with generators written as
$T'_{s'}$ for $s'\in S'$. 

\begin{Question}
Does the assignment $T'_{w_is_iw^{-1}_i} \mapsto T_{w_i} T_{s_i} T^{-1}_{w_i}$ for all $w_i s_i w^{-1}_i \in S'$ define an algebra embedding
$H_{W', S'} \to H_{S}$? 
\end{Question}

In other words, the question asks whether the subalgebra of $H_S$ generated 
by the elements $T_{w_i} T_{s_i} T^{-1}_{w_i}$ for all $w_i s_i w^{-1}_i\in S'$ 
is a Hecke algebra and, if so, the Hecke algebra associated with $(W', S')$.



{\bf Declaration of generative AI and AI-assisted technologies in the writing process}

During the preparation of this work the author(s) used ChatGPT in order to tweak the Introduction. After using this tool/service, the author(s) reviewed and edited the content as needed and take(s) full responsibility for the content of the publication.

\end{document}